\documentclass[12pt]{amsart}
\usepackage{amssymb}
\usepackage[all]{xypic}

\renewcommand{\labelenumi}{(\alph{enumi})}

\theoremstyle{plain}
\newtheorem{theorem}{Theorem}[section]
\newtheorem{lemma}[theorem]{Lemma}
\newtheorem{corollary}[theorem]{Corollary}
\newtheorem{prop}[theorem]{Proposition}

\theoremstyle{remark}

\newtheorem{remark}[theorem]{Remark}

\newtheorem*{note*}{Note}
\newtheorem*{remark*}{Remark}
\newtheorem*{example*}{Example}

\theoremstyle{definition}
\newtheorem*{definition*}{Definition}
\newtheorem{definition}[theorem]{Definition}

\newcommand{\Z}{\mathbb{Z}}

\newcommand{\Q}{\mathbb{Q}}
\newcommand{\C}{\mathbb{C}}
\newcommand{\N}{\mathbb{N}}
\newcommand{\F}{\mathbb{F}}

\newcommand{\Image}{\mathrm{Img}}
\newcommand{\Ker}{\mathrm{Ker}}
\newcommand{\Gal}{\mathrm{Gal}}
\newcommand{\tensor}{\otimes}

\newcommand{\Hom}{\mathrm{Hom}}

\newcommand{\disc}{\mathrm{Disc}}
\newcommand{\Norm}{\mathrm{Norm}}
\newcommand{\Cl}{\mathrm{Cl}}

\newcommand{\tr}{\mathrm{Tr}}
\newcommand{\OL}{\mathcal{O}_{L}}
\newcommand{\OK}{\mathcal{O}_{K}}
\newcommand{\OF}{\mathcal{O}_{F}}
\newcommand{\Ok}{\mathcal{O}_{k}}
\newcommand{\Frakm}{\mathfrak{M}}
\newcommand{\Frakp}{\mathfrak{P}}
\newcommand{\Frakq}{\mathfrak{Q}}
\newcommand{\Frakr}{\mathfrak{R}}
\newcommand{\frakp}{\mathfrak{p}}
\newcommand{\frakq}{\mathfrak{q}}
\newcommand{\frakr}{\mathfrak{r}}
\newcommand{\fraka}{\mathfrak{a}}
\newcommand{\frakb}{\mathfrak{b}}
\newcommand{\sseq}{\subseteq}
\newcommand{\Ram}{\mathrm{Ram}}
\newcommand{\Hone}{\mathrm{H}^{1}}
\newcommand{\eps}{\varepsilon}
\newcommand{\M}{\mathrm{M}}
\newcommand{\ind}{\mathrm{ind}}
\newcommand{\Pic}{\mathrm{Pic}}
\newcommand{\pic}{\mathrm{pic}}

\renewcommand{\O}{\mathcal{O}}

\renewcommand{\labelenumi}{(\alph{enumi})}

\title[Non-existence and splitting theorems for normal integral bases]
{Non-existence and splitting theorems \\ for normal integral bases}

\author{Cornelius Greither}
\address{Cornelius Greither\\
Fakult\"at f\"ur Informatik\\
Institut f\"ur theoretische Informatik und Mathematik\\
Universit\"at der Bundeswehr M\"unchen\\
85577 Neubiberg\\
Germany}
\email{cornelius.greither@unibw.de}

\author{Henri Johnston}
\address{Henri Johnston\\ 
St. Hugh's College\\
St. Margaret's Road\\
Oxford OX2 6LE\\
U.K.
}
\email{henri@maths.ox.ac.uk}
\urladdr{http://people.maths.ox.ac.uk/henri}

\thanks{Johnston was partially supported by a grant from the
 Deutscher Akademischer Austausch Dienst.}
\subjclass[2000]{Primary 11R33; Secondary 11R18}
\keywords{normal integral basis, additive Galois module structure}
\date{10th February, 2009}

\begin{document}

\begin{abstract}
We establish new conditions that prevent the existence of (weak) normal integral 
bases in tame Galois extensions of number fields. 
This leads to the following result: under appropriate technical hypotheses, the
existence of a normal integral basis in the upper layer of an abelian tower
$\Q \subset K \subset L$ forces the tower to be split in a very strong sense.
\end{abstract}

\maketitle

\section{Introduction}\label{intro}

Let $L/K$ be a tame abelian extension of number fields with Galois group $G$. 
Then the ring of integers $\OL$ is projective over the group ring $\OK[G]$
and we say that $L/K$ has a \emph{normal integral basis} (NIB) if $\OL$ is in fact 
free over $\OK[G]$. In the case $K=\Q$, the Hilbert-Speiser Theorem says that 
$L/K$ always has an NIB. However, the situation is rather more complex when $K \neq \Q$,
as illustrated by the following two results of Brinkhuis.

We call a number field $K$ a CM-field if it is a totally imaginary quadratic extension of a totally 
real field. Note that if $K/\Q$ is abelian then $K$ is either CM or totally real.

\begin{theorem}[\cite{brink-embed}]\label{br-non-split}
Let $K$ be a number field that is either CM or totally real and let $L$ be a finite abelian extension of $K$ of odd order.
Assume that for some subfield $k$ of $K$, over which $K$ and $L$ are Galois, the
short exact sequence of Galois groups
\[
1 \longrightarrow \Gal(L/K) \longrightarrow \Gal(L/k) \longrightarrow \Gal(K/k) \longrightarrow 1 
\]
is non-split. Then $L/K$ has no normal integral basis.
\end{theorem}

\begin{theorem}[\cite{brink}]\label{br-unram}
Let $L/K$ be an unramified abelian extension of number fields, each of which is either CM 
or totally real. If the Galois group of $L/K$ is not $2$-elementary, then $L/K$ has no normal integral basis.
\end{theorem}

A further obstruction rests on the factorization of resolvents and prevents 
the existence of so-called \emph{weak normal integral bases} (WNIBs). We recall
that $L/K$ has a WNIB if $\mathfrak{M}\OL$ is free over $\mathfrak{M}$, where
$\mathfrak{M}$ is the maximal $\OK$-order in the group algebra $K[G]$. 
In \cite{brink}, Brinkhuis shows non-existence of a WNIB in certain cases
when $L=\Q(\zeta_{p})$, the $p$th cyclotomic field, and Cougnard
generalises these results in \cite{cougnard}.

In Section \ref{non-existence} of this paper, we exhibit further cases in which there is no WNIB. 
The main technical difference with the work of Brinkhuis and Cougnard is that we do not
use comparison with absolute extensions when showing that certain resolvents
have nontrivial class.

Recall that two number fields $L$ and $K$ are said to be \emph{arithmetically disjoint} over 
a common subfield $F$ if $\O_{LK} = \OL \tensor_{\OF} \OK$, or equivalently, 
$(\disc(\OL/\OF),\disc(\OK/\OF))=\OF$ and $L$ is linearly disjoint from $K$ over $F$ (see \cite[III.2.13]{ft}).
In his classic book \cite{frohlich-book}, Fr\"ohlich makes the observation that
if $L$ and $K$ are arithmetically disjoint over $F$ and $L/F$ has an NIB, then 
so does $LK/K$. He goes on to say, ``What one wants
are of course somewhat less trivial conditions [for the existence of NIBs]''. 
In Section \ref{splitting-theorems}, we show that in certain settings there are
no such conditions! More precisely, we prove that under appropriate technical hypotheses,
the existence of an NIB in the upper layer of an abelian tower 
$\Q \subset K \subset L$ forces the tower to be \emph{arithmetically split}, that is, there
exists $L'/\Q$ arithmetically disjoint from $K$ over $\Q$ such that $L=L'K$.

\section{Preliminaries}

Let $L/K$ be a tame abelian extension of number fields with Galois group $G$. 
(In this paper, we take ``tame'' to mean ``at most tamely ramified''.) The group algebra 
$K[G]$ contains a unique maximal $\OK$-order $\Frakm=\Frakm_{K[G]}=\Frakm_{L/K}$.
We say that $L/K$ has a \emph{weak normal integral basis} (WNIB) if the projective
$\Frakm$-module $\Frakm \tensor_{\OK[G]} \OL$ is free. 
Note that we may identify $\Frakm \tensor_{\OK[G]} \OL$ with $\Frakm \OL \subset L$.

Let $D(G)=D(K,G)$ denote the set of $K$-irreducible characters of $G$. 
For each $\psi \in D(G)$, let $D(\psi)$ denote the set of absolutely irreducible characters $\chi$ 
such that $\psi = \sum_{\chi \in D(\psi)} \chi$. Let 
\[ e_{\psi} := \frac{1}{|G|}\sum_{g \in G} \psi(g^{-1})g \]
be the corresponding primitive idempotent of $K[G]$. 
For $\psi \in D(G)$, fix an absolutely irreducible character 
$\chi_{\psi} \in D(\psi)$ and let $K_{\psi}=K(\chi_{\psi})$
be the field extension of $K$ generated by the values of $\chi_{\psi}$.
(Note that $K_{\psi}$ does not depend on the choice of $\chi_{\psi} \in D(\psi)$.)
Then we have $K$-algebra isomorphisms
$ \Phi_{\psi}: K[G] e_{\psi} \longrightarrow K_{\psi} $
induced by $g \mapsto \chi_{\psi}(g)$. These restrict to isomorphisms 
$\Frakm_{\psi} := \Frakm e_{\psi} \longrightarrow \mathcal{O}_{K_{\psi}}$
and we have
\[ K[G] = \bigoplus_{\psi \in D(G)} K[G] e_{\psi} 
\cong \bigoplus_{\psi \in D(G)} K_{\psi} \quad \textrm{and} \quad
\Frakm = \bigoplus_{\psi \in D(G)} \Frakm e_{\psi} \cong \bigoplus_{\psi \in D(G)} \mathcal{O}_{K_{\psi}}.\]
For $\chi \in D(\C, G)$ and $\alpha \in L$, define 
\[ ( \alpha \mid \chi )
= ( \alpha \mid \chi )_{L/K} := \sum_{g \in G} \chi(g^{-1}) g(\alpha) = 
|G|e_{\chi}\alpha \in L(\chi) \]
to be the \emph{resolvent} attached to $\alpha$ and $\chi$. 
Denote by $( \OL : \chi )$ the $\mathcal{O}_{K(\chi)}$-module generated by 
the $(\alpha \mid \chi)$ with $\alpha \in \OL$ (note that there exists a unique 
$\psi \in D(G)$ such that $\chi \in D(\psi)$, and 
$\mathcal{O}_{K(\chi)} = \mathcal{O}_{K_{\psi}}$ acts via $\Phi_{\psi}^{-1}$).

\begin{prop}\label{WNIB-resolvent}
$L/K$ has a WNIB if and only if $(\OL : \chi_{\psi})$ is free over $\mathcal{O}_{K_{\psi}}$
for every $\psi \in D(G)$. (Note that this is true irrespective of the choices of $\chi_{\psi} \in D(\psi)$.)
\end{prop}

\begin{proof}
We emulate the argument given for \cite[Proposition 1.2]{greither-RNIB}. Observe that
\begin{eqnarray}
& & \Frakm \tensor_{\OK[G]} \OL \cong \Frakm \textrm{ as } \Frakm \textrm{-modules} \\
& \Longleftrightarrow &  \Frakm e_{\psi} \tensor_{\OK[G]} \OL \cong \Frakm e_{\psi} \textrm{ as } \Frakm e_{\psi}
\textrm{-modules for all } \psi \in D(G) \\
& \Longleftrightarrow & \mathcal{O}_{K_{\psi}} \tensor_{\OK[G]} \OL \cong \mathcal{O}_{K_{\psi}} 
 \textrm{ as } \mathcal{O}_{K_{\psi}} \textrm{-modules for all } \psi \in D(G).
\end{eqnarray}
Therefore it suffices to show that 
$\mathcal{O}_{K_{\psi}} \tensor_{\OK[G]} \OL \cong (\OL : \chi_{\psi})$ for each $\psi \in D(G)$.
Consider the map $\varphi : \OL \rightarrow  (\OL : \chi_{\psi})$, 
$\alpha \mapsto (\alpha \mid \chi_{\psi})$.
We let $G$ operate on $\mathcal{O}_{K_{\psi}}$ via $\chi_{\psi}$, hence $\varphi$ is $\OK[G]$-linear,
and we obtain an epimorphism
\[ \varphi' : \mathcal{O}_{K_{\psi}} \tensor_{\OK[G]} \OL \longrightarrow (\OL : \chi_{\psi}) .\]
By a rank argument, $\varphi'$ is also injective.
\end{proof}

\begin{corollary}\label{resolvent-not-principal}
Let $L'$ be any finite extension of $L$ such that $L(\chi) \sseq L'$ for all $\chi \in D(\C,G)$.
If $L/K$ has a WNIB, then for every $\psi \in D(G)$ the ideal 
$\mathcal{O}_{L'} (\OL : \chi_{\psi})$ is principal.
\end{corollary}

\begin{proof}
This follows trivially from Proposition \ref{WNIB-resolvent} once one notes that the hypothesis 
ensures $(\OL : \chi_{\psi}) \sseq \mathcal{O}_{L'}$ for every $\psi \in D(G)$.
\end{proof}

\begin{lemma}\label{trace-down}
Let $K \subset L \subset N$ be a tower of number fields such that $N/K$ is tame abelian. 
If $N/K$ has an NIB (resp.\ WNIB), then $L/K$ also has an NIB (resp.\ WNIB).
\end{lemma}

\begin{proof}
Let $G=\Gal(N/K)$ and $H=\Gal(L/K)$. 
Since $N/L$ is tame, $\tr_{N/L}(\mathcal{O}_{N}) = \mathcal{O}_{L}$.
Suppose that $N/K$ has an NIB, i.e.,
there exists $\alpha \in \mathcal{O}_{N}$ such that 
$\mathcal{O}_{N} = \mathcal{O}_{K}[G] \cdot \alpha$.
Adapting the proof of \cite[Lemma 6]{byott-lettl}, we have 
\[
\mathcal{O}_{L} = \tr_{N/L}(\mathcal{O}_{N})
= \tr_{N/L}(\mathcal{O}_{K}[G] \cdot \alpha) 
= \mathcal{O}_{K}[G] \cdot \tr_{N/L}(\alpha)
= \mathcal{O}_{K}[H] \cdot \tr_{N/L}(\alpha).
\]
A similar argument applies for WNIBs. 
\end{proof}

\begin{lemma}\label{arith-disjoint}
Let $K$ be a number field with finite extensions $L$ and $F$ such that $L/K$ is tame abelian.
Let $E=LF$ and suppose that $L$ and $F$ are arithmetically disjoint over $K$.
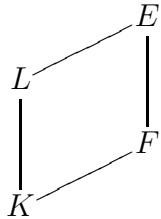
$$
\xymatrix@1@!0@=24pt { 
& & E \\
L \ar@{-}[urr] & & \\
& & F \ar@{-}[uu] \\
K \ar@{-}[uu] \ar@{-}[urr] & & \\
}
$$
If $L/K$ has an NIB (resp.\ WNIB), then $E/F$ also has an NIB (resp.\ WNIB).
\end{lemma}

\begin{proof}
Straightforward.
\end{proof}

\section{Bounding a certain kernel}

\begin{definition}
Let $L/K$ be Galois extension of number fields with Galois group $G$. 
\begin{enumerate}
\item Let $\Ram(L/K)$ be the set of finite primes of $K$ that ramify in $L$.
\item For $\frakp \in \Ram(L/K)$, let $e_{\frakp}=e_{\frakp, L/K}$ 
denote the ramification index of $\frakp$ in $L/K$.\\
(Note that $e_{\frakp}$ is well defined because $L/K$ is Galois.)
\item Let $\M(L/K)$ be the abelian group $\bigoplus_{\frakp \in  \Ram(L/K)} \Z / e_{\frakp} \Z$.
\item Define a homomorphism of abelian groups
\[ \varepsilon_{L/K} : \M(L/K) \longrightarrow \frac{\Cl(\OL)^{G}}{\Image(\Cl(\OK))} \] 
by sending $\bar 1$ (at position $\frakp$)
to the class of $\prod_{\Frakp|\frakp}\Frakp$ ($\Frakp$ prime of $L$) where
$\Image(\Cl(\OK))$ denotes the image of the natural map 
$\Cl(\OK) \longrightarrow \Cl(\OL)$.\\
(Note that $(\prod_{\Frakp|\frakp}\Frakp)^{e_\frakp} = \frakp\OL$,
and therefore $\varepsilon_{L/K}$ is well defined.)
\end{enumerate}
Suppose that $L$ is a CM-field and that $K$ is either CM or totally real.  
\begin{enumerate}
\setcounter{enumi}{4}
\item Let $j$ denote complex multiplication. 
\item For a finite $G$-module $X$, let $X_{\mathrm{odd}}$ be the odd part
of $X$. If $j$ also acts on $X$, let $X^{-}$ be the minus part of $X_{\mathrm{odd}}$, i.e., $X^{-} = (X_{\mathrm{odd}})^{1-j}$.
\item Let $\mu_{L}$ be the group of roots of unity in $L$.
\end{enumerate}
\end{definition}

We make no claim that the following result is new; for cyclic
$G$ it can be deduced from results in \cite[Chapter 13, \S4]{lang-cyclo-II}.

\begin{prop}\label{ambig}
Suppose that $j$ commutes with every element of $G$. Then $j$ acts on $\Ram(L/K)$ and
$\Ker (\varepsilon_{L/K})^{-}$ is isomorphic to a subquotient of $\Hone(G,\mu_{L,\mathrm{odd}})$.
\end{prop}

\begin{proof}
The first claim is immediate.

Let $I$ be any ambiguous ($G$-stable) ideal of $\OL$ such 
that $[I] \in \Cl(\OL)^{-} \cap \Image(\Cl(\OK))$.
Then $I = x\fraka \OL$ for some $x \in \OL$ and 
some ideal $\fraka$ of $\OK$. Thus, $J:=I^{1-j}= y\fraka^{1-j} \OL$ with
 $y=x^{1-j}$, so $y$ is an anti-unit, i.e.,
 $y^{1+j}=1$. For every $\sigma \in G$,
$J^\sigma=J$, hence $y^{\sigma-1}$ must be a unit of $L$,
and therefore a root of unity because it is an anti-unit as well.
The map $\alpha_I: \sigma \mapsto y^{\sigma-1}$ is a 1-cocycle
on $G$ with values in $\mu_L$. 

For $z=(z_\frakp)_\frakp \in \M(L/K)$, let $I(z)$ denote
the ideal $\prod_\frakp (\prod_{\Frakp|\frakp} \Frakp)^{z_\frakp}$,
so that $\eps_{L/K}(z)$ is  the class of $I(z)$. 
Now assume $z \in \Ker (\varepsilon_{L/K})^{-}$. 
Then $J=I(z)^{1-j}$ can be written as above and one obtains 
a cocycle $\alpha_{I(z)}$. Let $\bar{\alpha}(z)$ denote the class of this cocycle in
$\Hone(G,\mu_L)$. If $\bar{\alpha}(z)$ is trivial, then
there exists a root of unity $\zeta$ such that $\zeta^{1-\sigma}=
\alpha_{I(z)}(\sigma) = y^{\sigma-1}$ for all $\sigma\in G$.
Then setting $y' = \zeta y$ gives $J = y'\fraka^{1-j}\OL$, and
$y'$ is fixed under $G$, hence in $K$. Therefore
$J$ is induced from an ideal of $\OK$. Now $J = I((1-j)z) =
I(2z)$; from the definition  one sees that $I(z')$
is only induced from a $\OK$-ideal if $z'$ is trivial in $\M(L/K)$.
Hence we have $z=0$. 

It remains only to resolve the technical problem that $\bar{\alpha}$
is not necessarily a homomorphism.
Let $U$ be the group of all cocycles $w\mapsto w^{\sigma-1}$,
where $w$ is an anti-unit generating an ideal $\frakb^{1-j}\OL$
with $\frakb$ an ideal of $\Ok$.
 Then changing the representation
$I=x\fraka\OL$ to another $I=x_1\fraka_1\O_L$ changes $\bar{\alpha}_I$
by a factor in $U$. Conversely, if $\bar{\alpha}_I \in U$, we
can change the representation of $I$ so as to make $\bar{\alpha}_I$ trivial.
If we define $\beta$ to be $\bar{\alpha}$ followed by the
projection $\Hone(G,\mu_L) \to \Hone(G,\mu_L)/U$, we see
that $\beta$ is an injective homomorphism.
Since the domain of definition of $\beta$ has odd order, 
we may replace $\mu_L$ by its odd part. 
\end{proof}

\section{Non-existence of weak normal integral bases}\label{non-existence}

\begin{definition}\label{abuse-ram-def}
Let $k \subset K \subset L$ be a tower of number fields. We adopt the following 
harmless abuse of language: we say that a prime $\frakp$ of $k$ ramifies in 
$L/K$ if some prime above $\frakp$ ramifies in $L/K$. We denote by $e_{\frakp,K/k}$
the ramification degree of $\frakp$ in $K/k$ and, if $L/k$ is Galois, we let $e_{\frakp,L/K}$ 
denote the ramification degree in $L/K$ of \emph{any} prime above $\frakp$ in $K$.
\end{definition}

\begin{theorem}\label{nownib1} 
Let $L/k$ be an abelian extension of number fields.
Let $K$ be an intermediate field such that $L/K$ is tame
and $K$ is totally real. Suppose there exists a prime $\frakp$ of $k$ such that
$e_{\frakp,K/k}$ has a nontrivial odd factor and $e_{\frakp,L/K}$ has an odd prime factor $\ell$,
for which the following two conditions are satisfied:
\begin{enumerate}
\item $K$ is linearly disjoint from $\Q(\zeta_{\ell})$ over $\Q$ (equivalently, $[K(\zeta_{\ell}):K]=\ell-1$); and
\item if $\ell=3$, then $e_{\frakp,K/k}$ has an odd prime divisor $q$ such that
$\zeta_{q} \not \in L(\zeta_{3^{\infty}})$.
\end{enumerate}
Then $L/K$ has no WNIB.
\end{theorem}

\begin{proof}
(1) Looking at ramification groups, we see that there exists an intermediate
extension $K \subset \tilde{L} \subset L$ such that $\tilde{L}$ is cyclic of $\ell$-power
degree over $K$ and in which (a prime above) $\frakp$ is ramified with precise exponent $\ell$.
By Lemma \ref{trace-down}, we may therefore suppose without loss of generality that in fact
$L = \tilde{L}$. (We have implicitly used the hypothesis that $L/k$ is abelian here; if we were only
to assume $L/k$ Galois and $L/K$ abelian, then $\tilde{L}/k$ would not necessarily be Galois.)
Note that $L$ is totally real since $[L:K]$ is odd,  $L/K$ is Galois and $K$ is totally real.
Furthermore, since $[L:K]$ is a power of $\ell$, $[K(\zeta_{\ell}):K]=\ell-1$ implies that
$[L(\zeta_{\ell}):L]=\ell-1$. In other words, $L$ is linearly disjoint from $\Q(\zeta_{\ell})$ over $\Q$.

(2) There exists an intermediate extension $k \subset \tilde{k} \subset K$ such that (a prime above)
$\frakp$ is still ramified in $K/\tilde{k}$ and $K/\tilde{k}$ is cyclic of odd prime order. 
(In the case $\ell=3$, we choose $\tilde{k}$ such that $[K:\tilde{k}]=q$.)
We may therefore suppose without loss of generality that in fact $k=\tilde{k}$.

(3) Let $r=[K:k]$ (an odd prime), $s=[L:K]$ (a power of $\ell$) and $L'=L(\zeta_{s})$.
By Corollary \ref{resolvent-not-principal}, it is enough to show for some faithful character
$\chi: \Gal(L/K) \rightarrow L'^{\times}$ that the ideal $I:=\mathcal{O}_{L'}(\OL : \chi)$ is not principal.
From the fact that $\ell$ divides $e_{\frakp,L/k}$,
we know that $\ell$ also divides $\Norm_{k/\Q}(\frakp)-1=|(\Ok/\frakp)^{\times}|$
(use the local Artin map) and so $\frakp$ is totally split in
$k(\zeta_{\ell})/k$. Let $\Delta = \Gal(L(\zeta_{\ell})/L)$. Then $\Delta$ is canonically
isomorphic to a subgroup $D$ of $(\Z/\ell\Z)^{\times}$; we denote the automorphism
attached to $i \in D$ by $\delta_{i}$. Since $L$ is linearly disjoint from
$\Q(\zeta_{\ell})$ over $\Q$, we in fact have $D=(\Z/\ell\Z)^{\times}$. (Note that
it is possible to weaken the disjointness hypothesis - see Remark \ref{weaken-lin-dis}.)

Fix a prime $\Frakp$ of $L$ above $\frakp$ and 
let $\F$ be the residue field $\OK/\Frakp \cap K$.
Let $\eta$ be the restriction of the local Artin map, $\F^{\times} \rightarrow \Gal(L/K)$, followed by $\chi$. 
Then $\eta$ has image exactly $\mu_{\ell} \subset L(\zeta_\ell)$ because $e_{\frakp,L/K}=\ell$.
Define $\gamma:\F^{\times} \rightarrow \F^{\times}$ by $x \mapsto x^{-f}$ where 
$f=|\F^{\times}|/\ell$ (note that $|\F^{\times}|=\Norm_{K/\Q}(\Frakp \cap \OK)-1$ is a multiple of
$|(\Ok / \frakp)^{\times}| = \Norm_{k/\Q}(\frakp)-1$, which is divisible by $\ell$).
It is straightforward to see that $\gamma$ has image  $\mu_\ell \subset \F^{\times}$.
Therefore there exists exactly one prime ideal
$\Frakq$ above $\Frakp$ in $L(\zeta_\ell)$ such that $\eta$ and
$\gamma$ agree modulo $\Frakq$. From \cite[Theorem 26 (i)]{frohlich-book} 
(the proof of which is a fairly standard argument resting crucially on a certain local 
calculation involving a Kummer extension) we obtain:
\begin{enumerate}
\renewcommand{\labelenumi}{(\Alph{enumi})}
\item For every prime $\Frakr$ of $L'$ above $\Frakq$, 
the ideal $I=\mathcal{O}_{L'}(\OL : \chi)$ has valuation $1$ at  $\Frakr$.
\end{enumerate}
Since $\frakp$ is totally split in $k(\zeta_{\ell})/k$, 
we know that $\Frakp$ splits into $\ell - 1$ factors in $L(\zeta_{\ell})$; 
these are permuted by $\Delta=\Gal(L(\zeta_{\ell})/L)$. From loc.\ cit., we also obtain:
\begin{enumerate}
\renewcommand{\labelenumi}{(\Alph{enumi})}
\setcounter{enumi}{1}
\item For all $i=1, \ldots, \ell - 1$ and every prime $\Frakr$ of $L'$ above $\delta_{i}^{-1}\Frakq$,
the ideal $I=\mathcal{O}_{L'}(\OL : \chi)$ has valuation $i$ at  $\Frakr$.
\end{enumerate}
\[
\xymatrix@1@!0@=30pt {
& & & & \Frakr \\
& & \Frakq  & & L' \ar@{-}[dll]  \\
\Frakp & & L(\zeta_{\ell})  \ar@{-}[dll]_{\Delta} \ar@{-}[dd] & & \\
L \ar@{-}[dd]_{s} & & & &  \\
& & K(\zeta_{\ell}) \ar@{-}[dd] \ar@{-}[dll]  & &\\
K \ar@{-}[dd]_{r} & & & &  \\
& & k(\zeta_{\ell}) \ar@{-}[dll] & & \\
k & & & &\\
\frakp & & & &\\
}
\]

Now let $\Frakq_i$ denote the product of all primes of $L'$ 
over $\delta_{i}^{-1}\Frakq$. Then $\Frakp \mathcal{O}_{L'} = \Frakq_1 \cdots \Frakq_{\ell-1}$,
where the factors are pairwise coprime. Let
\[ \theta= \sum_{i=1}^{\ell-1} i \delta_i^{-1} \in \Z[\Delta] .\]
By definition of the $\Frakq_i$, the Galois group
$\Gal(L'/L)$ acts on them through its quotient $\Delta$. Thus 
the expression $\Frakq_1^\theta$ makes sense, and from (B) we obtain
the following key information:
\begin{enumerate}
\renewcommand{\labelenumi}{(\Alph{enumi})}
\setcounter{enumi}{2}
\item The ``above-$\Frakp$-part'' of $I=\mathcal{O}_{L'}(\OL : \chi)$ is $\Frakq_{1}^{\theta}$.
\end{enumerate}

(4) We need two auxiliary fields: let $F$ be the inertia field of
$\frakp$ in $L/k$ and put $F'=F(\zeta_{s})$. Then $[L:F]=r\ell$
and $\frakp$ is totally ramified in $L/F$. If $r$ and $\mathfrak{p}$ are coprime,
then since $\frakp$ is tamely ramified in $L/K$, it is also tamely ramified in $L/F$
and so $\Gal(L/F)$ must be cyclic. If $r$ and $\mathfrak{p}$ are not coprime,
then $\frakp$ has wild ramification degree $r$ in $L/F$ and so
$\Gal(L/F)$ is isomorphic to a semi-direct product of $(\Z/\ell\Z)$ with $(\Z/r\Z)$.
However, the hypothesis that $L/k$ is abelian forces this product to be direct
and so again $\Gal(L/F)$ is cyclic (note that if $r=\ell$, then $r$ is coprime to $\frakp$).
Furthermore, $L$ must be linearly disjoint from $F'$ over $F$ because $F'/F$ is unramified
at $\frakp$, whereas $L/F$ is totally ramified at $\frakp$.
Let  $\frakq_i$ be the product of all distinct primes of $F'$ below factors of $\Frakq_i$.
(Note that because of total ramification, $\frakp$-primes in $F'$ and $L'$ correspond bijectively.) 
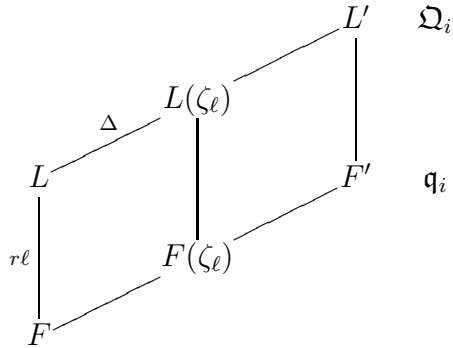
\[
\xymatrix@1@!0@=30pt {
& & & & L' \ar@{-}[dll]  \ar@{-}[dd]  & \Frakq_{i}  \\
& & L(\zeta_{\ell})  \ar@{-}[dll]_{\Delta} \ar@{-}[dd] & & & \\
L \ar@{-}[dd]_{r\ell} & & & & F' \ar@{-}[dll]  & \frakq_{i} \\
& & F(\zeta_{\ell}) \ar@{-}[dll]  & & & \\
F & & & & & \\
}
\]
From the definition of resolvents, we see that $I$ is an ambiguous ideal 
under $\Gal(L'/F')$. 

(5) In taking minus parts in what follows, it is important to note that complex 
conjugation $j$ is just $\delta_{-1}$ because $L$ is totally real.
Since $\Gal(L'/F')$ is cyclic, we find that $\Hone(L'/F',\mu_{L'})=\Hom(L'/F',\mu_{L'})$ 
is also cyclic. By Proposition \ref{ambig} (applied to $L'/F'$ instead of $L/K$), we see 
that $\Ker(\eps_{L'/F'})^{-}$ is a cyclic group.

The group $\Delta=\Gal(L(\zeta_{\ell})/L)$, which can be seen as the non-$\ell$-part 
of $\Gal(L'/L)$, acts on both $F'/F$ and $L'/L$. Let $R=(\Z/r\ell\Z)[\Delta]$.
Then $\eps_{L'/F'}$ is in fact an $R$-module homomorphism.
We can write $\M(L'/F')=A \oplus B$ where $A$ is the $R$-module consisting of 
elements $x=(x_{\frakr})_{\frakr}$ with $x_{\frakr}=0$ for all $\frakr \not \in \{ \frakq_{i} \}$
and $B$ is the $R$-module consisting of elements $y=(y_{\frakr})_{\frakr}$ with
$y_{\frakq_{i}}=0$ for all $\frakq_{i}$. 
By abuse of notation, we do not distinguish between $\theta \in \Z[\Delta]$
and its projection to $R$.   Let $z \in A$ denote the element with entries $1$ 
at all primes dividing $\frakq_1$, and zeros elsewhere. Then $\theta z \in A$
and by the partial factorization given in (C), 
we know that there exists $z' \in B$ such that
\[ [I] = \eps_{L'/F'}(\theta z + z'). \]

We now proceed by contradiction. Suppose that $L/K$ does in fact have a WNIB.
Then by Corollary \ref{resolvent-not-principal}, the resolvent ideal $I$ must be principal
and so $\theta z + z' \in \Ker(\eps_{L'/F'})$. Let $\pi : A \oplus B \rightarrow A$
be the natural projection. Then 
\[
\pi(\theta z + z')  = \theta z \in \pi(\Ker(\eps_{L'/F'}))
\]
and so letting $J=R\theta$ be the ideal of $R$ generated by $\theta$, we have
\[
J^{-}z \sseq \pi(\Ker(\eps_{L'/F'}))^{-}=\pi(\Ker(\eps_{L'/F'})^{-}).
\]
However,
$\pi(\Ker(\eps_{L'/F'})^{-})$ is cyclic as an abelian group since the same is true
for $\Ker(\eps_{L'/F'})^{-}$ (note that our group $G$ is cyclic by
the reduction performed in (1) and use Proposition \ref{ambig}). Therefore in order to show that $L/K$ has no WNIB, 
it suffices to show that $(Jz)^{-}=J^{-}z$ is not cyclic as an abelian group. 
Since $Rz$ is a free $R$-submodule of $A$ of rank $1$, we see that 
$J^{-}$ and $J^{-}z$ are isomorphic as $R$-modules and hence as abelian groups. 
Thus we are further reduced to showing that $J^{-}$ is not cyclic as an abelian group.

(6) Assume that $\ell > 3$ (we shall return to the $\ell=3$ case later).
We consider the two elements $\ell\theta$ and $(2-\delta_2)\theta$
in $\ell R$, which identifies with $(\Z/r\Z)[\Delta]$
(``division by $\ell$''). Then the two elements take the shape
\[ u =  1\cdot\delta_1^{-1} + 2\cdot\delta_2^{-1} + \ldots + (\ell-1)\cdot\delta_{\ell-1}^{-1} \]
and
\[v = \delta_{(\ell+1)/2}^{-1}+\delta_{(\ell+3)/2}^{-1}+ \ldots + \delta_{\ell-1}^{-1}, \]
respectively. We now project them into the minus part, by sending
$\delta_{\ell-i}^{-1}$ to $-\delta_{i}^{-1}$ for $i=1,\ldots,(\ell-1)/2$.
The result is
\[ u_- = (2-\ell)\cdot\delta_1^{-1} + (4-\ell)\cdot\delta_2^{-1} + 
\ldots + (-1)\cdot \delta_{(\ell-1)/2}^{-1},\]
and
\[ v_- = -\delta_{1}^{-1}-\delta_{2}^{-1}- \ldots - \delta_{(\ell-1)/2}^{-1}. \]
Looking just at the first two coefficients of $u_-$ and $v_-$
and noting that 
\[
\det \left( 
\begin{array}{cc}
2-\ell & 4-\ell \\
-1 & -1
\end{array}
\right)
= \ell-2+4-\ell=2,
\]
we see that $u_-$ and $v_-$ between them generate an abelian group
of type $(r,r)$. In particular, $J^-$ cannot be cyclic as an abelian group.

(7) Finally, we discuss the case $\ell=3$. 
We have $[K:k]=r=q$ by the choice made in step (2).
By condition (b), $\zeta_{q} \not \in L' \subset L(\zeta_{3^{\infty}})$ and so
the group $\Hone(L'/F',\mu_{L'})$ has order prime to $q$.
Hence by Proposition \ref{ambig}  (applied to $L'/F'$ instead of $L/K$), 
we see that $\Ker(\eps_{L'/F'})^{-}$ also has order prime to $q$.
However, the element $\theta$ projected to the minus part of $(\Z/3q\Z)[\Delta]$
comes out as $(2-\ell)\delta_1=-\delta_1$, which has order $q\ell$. 
The argument is completed as before. 
\end{proof}

We now give a corollary that will used in the proof of Theorem \ref{NIB-split}.
For this we need a compatibility result for resolvents, which the authors were unable to find in the literature, 
but seems unlikely to be new. We give a proof for the convenience of the reader.

We retain the notation $K \subset \tilde L \subset L$ from step (1) in the above proof, 
dropping the assumption that $L$ equals $\tilde L$.
To $\tilde L$ we associated a resolvent ideal which we now write $\tilde I = \mathcal{O}_{L'}
(\mathcal{O}_{\tilde L}:\tilde \chi)$ with $\tilde \chi$ a faithful character of
$\tilde L/K$. We likewise have a resolvent ideal $I = \mathcal{O}_{L'}
(\mathcal{O}_{ L}:\chi)$ for any character $\chi$  of
$L/K$. (A choice for $\chi$ will be made in a moment.)
Now assume $L/K$ is cyclic and write $t$ for the degree $[L:\tilde L]$. Then by construction $L/\tilde L$
is totally ramified at all primes above $\frakp$; we pick a faithful character $\chi$
of $L/K$ such that $\chi^t$ induces $\tilde \chi$. 

\begin{lemma}\label{normcompat}   
Under the conditions above, we have the following norm compatibility:
\[ (\hbox{\rm N}_{L'/{\tilde L}'}I)_\frakp = \tilde I_\frakp, \]normalized
where the subscript $\frakp$ denotes taking the above-$\frakp$ part of an ideal,
that is, one omits all powers of prime ideals not above $\frakp$ from the factorization
of the ideal.
\end{lemma}

\begin{proof} The assertion is equivalent to the corresponding assertion for all completions
at places above $\frakp$. So for the rest of this proof we assume that all our fields are complete
(the base field $k$ is replaced by its $\frakp$-adic completion), but 
denote them by the same letters as before. 

As in the proof of Theorem \ref{nownib1}, we look at  
\cite[p.135]{frohlich-book} ($F$ there being our $K$),
and we consider the characters $\eta$ and $\tilde\eta$
of $\OK^{\times}$ afforded via local class field theory by
$\chi$ and $\tilde\chi$ respectively. (Fr\"ohlich's
notation is $\varphi$ instead of $\eta$.) Then we have
$\tilde\eta=\eta^t$, and therefore the integers $s$ and
$\tilde s$ attached to $\eta$ and $\tilde\eta$ resp.\
as in loc.\ cit.\ are also linked by the relation $\tilde s =ts$.
From this and again \cite[Theorem 26 (i)]{frohlich-book} we obtain,
with $v$ the normalized $p$-adic valuation:
\[ v(\tilde I) = t\cdot v(I). \]
Since the degree $t$ extension $L/\tilde L$ is totally ramified, 
the preceding formula amounts exactly to the required norm relation.
\end{proof}

\begin{corollary}\label{nownib1-cor}
Assume the hypotheses and notation of Theorem \ref{nownib1}
and make the further assumption that $L/K$ is cyclic of degree $\ell^{m}$
for some $m \geq 1$. Suppose that there exist fields $L_{1}$ and $K_{1}$ 
with $L \sseq L_{1} \sseq L(\zeta_{\ell})$,
$K \sseq K_{1} \sseq K(\zeta_{\ell})$ and $[L_{1}:L]=[K_{1}:K]$.
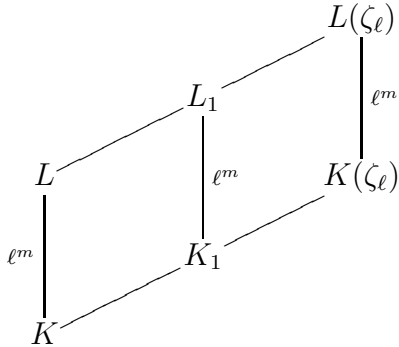
\[
\xymatrix@1@!0@=30pt {
& & & & L(\zeta_{\ell}) \ar@{-}[dll] \ar@{-}[dd]^{\ell^{m}}   \\
& & L_{1}  \ar@{-}[dll] \ar@{-}[dd]^{\ell^{m}} & &  \\
L \ar@{-}[dd]_{\ell^{m}} & & & & K(\zeta_{\ell}) \ar@{-}[dll]  \\
& & K_{1} \ar@{-}[dll]  & &  \\
K & & & & \\
}
\]
Then $L_{1}/K_{1}$ does not have a WNIB.
\end{corollary}

\begin{remark}
Note that Theorem \ref{nownib1} does not apply directly to $L_{1}/K_{1}$ because $K_{1}$
is not linearly disjoint from $\Q(\zeta_{\ell})$ over $\Q$ and is not necessarily totally real.
\end{remark}

\begin{proof}
$L$ is linearly disjoint from $K_{1}$ over $K$ since $[L:K]=\ell^{m}$ and $[K_{1}:K]$ divides
$\ell-1$. Furthermore, $L/K$ is tamely ramified and $K_{1}/K$ is only ramified at
primes above $\ell$, if at all. Therefore $L$ is in fact arithmetically disjoint from $K_{1}$
over $K$, and so for any nontrivial character $\chi$ of $\Gal(L_{1}/K_{1})\cong\Gal(L/K)$,
we have $\mathcal{O}_{L'}(\mathcal{O}_{L} : \chi) = \mathcal{O}_{L'}(\mathcal{O}_{L_{1}} : \chi)$
(note that  $L_{1}=K_{1}L$ and $L'=L(\zeta_{\ell^{m}})$ in this case).
It remains to prove that $I= \mathcal{O}_{L'}(\mathcal{O}_{L} : \chi)$ is not principal.
We closely follow the argument given in the proof of Theorem \ref{nownib1}, using 
Lemma \ref{normcompat}.

There are elements $y$ and $y_1$ in the above-$\frakp$ part (resp.\ the not-above-$\frakp$ part)
of $M(L'/F')$, such that $\varepsilon_{L'/F'}(y+y_1) = [I]$. Similarly (and as before) we
have $\tilde z$ and ${\tilde z}_1$ in the above-$\frakp$ part (resp.\ the not-above-$\frakp$ part)
of $M({\tilde L}'/F')$, such that $\varepsilon_{{\tilde L}'/F'}(\tilde z + \tilde z_1) = [\tilde I]$.
By Lemma \ref{normcompat} and step (5) in the proof of Theorem \ref{nownib1} we may choose 
$y$ and $\tilde z$ in such a way that 
$\theta z~=~\tilde z ~=~ \mathrm N_{L'/{\tilde L}'} y$. Then $y$ generates a noncyclic $\Z$-submodule
of $M(L'/F')^-$ since  $\theta z = \mathrm{N}_{L'/{\tilde L}'} y$
generates a noncyclic submodule of $M({\tilde L}'/F')^-$
as already shown. As at the end of step (5) in the proof of Theorem \ref{nownib1}, this implies that
$I$ is not principal (even more: the class of $I$ is nontrivial in the target of the
map $\varepsilon_{L'/F'}$). Hence $L'/F'$ has no WNIB, and as shown in the first paragraph of
the proof, this implies that $L_1/K_1$ does not have a WNIB either.
\end{proof}

We now discuss just a few of the many variants that Theorem \ref{nownib1} admits.

\begin{remark}\label{weaken-lin-dis}
Condition (a) of Theorem \ref{nownib1} requires that $K$ is linearly disjoint
from $\Q(\zeta_{\ell})$ over $\Q$, or equivalently, that $K(\zeta_{\ell})$ has
maximal degree $\ell-1$ over $K$. Recall from the proof that $D$ is defined to the 
subgroup of $(\Z/\ell\Z)^{\times}$ that is canonically isomorphic to 
$\Gal(L(\zeta_{\ell})/L)$. Minor modifications of the argument allow 
condition (a) to be replaced with a weaker, though more cumbersome, hypothesis:
\begin{enumerate}
\item[($\textrm{a}'$)]
\begin{enumerate}
\item $L$ is linearly disjoint from $k(\zeta_{\ell})$ over $k$; and
\item for $g$ some Fermat prime or $g=2$, we have $\ell > g^{2}$ and 
$\bar{g} \in D \subset (\Z/\ell\Z)^{\times}$.
\end{enumerate}
\end{enumerate}
Note that the only known Fermat primes are $3,5,17,257$ and $65537$.
Since $L$ is totally real and $L(\zeta_{\ell})$ is totally complex, $[L(\zeta_{\ell}):L]$ is even
and so we always have $\overline{-1} \in D$. Hence ($\textrm{a}'$) is no 
improvement over (a) when $\ell \leq 11$, but for example if $\ell=13$, then $D$ 
could be the subgroup of order $6$.

We briefly outline the necessary changes to step (6) of the proof of Theorem \ref{nownib1}.
Consider the elements $\theta$ and $(g-\delta_{g})\theta$ in $\ell R$, which we identify
with elements $u$ and $v$ in $(\Z/r\Z)[\Delta]$, as before. Then using the fact that $g^{2} < \ell$, we compute basis representations for $u_{-}$ and $v_{-}$. We have
\[
\det \left( 
\begin{array}{cc}
2-\ell & 2g-\ell \\
1-g & 1-g
\end{array}
\right)
= 2(g-1)^{2},
\]
where the entries of the upper (resp.\ lower) row of the matrix are
the coefficients of $u_-$ (resp.\ $v_-$) at $\delta_1^{-1}$ and $\delta_g^{-1}$. 
Since $g=2$ or is a Fermat prime, $2(g-1)^{2}$ is some power of $2$ and hence 
relatively prime to $r$ (some odd prime), and the proof concludes as before. Clearly, 
one could weaken ($\textrm{a}'$) even further in special cases where, for example, $r$ is known.
\end{remark}

\begin{theorem}\label{nownib2} 
Let $L/k$ be a Galois (not necessarily abelian) extension of number fields.
Let $K$ be an intermediate field such that $K/k$ and $L/K$ are abelian,
$K$ is totally real, and $L/K$ is tame of odd prime degree $\ell$.
Suppose there exists a prime $\frakp$ of $k$ that is (totally) ramified in $L/K$
(so $e_{\frakp,L/K}=\ell$) and $e_{\frakp,K/k}$ has a nontrivial odd factor.
Assume moreover that conditions (a) and (b) of Theorem \ref{nownib1} are satisfied.
Then $L/K$ has no WNIB.
\end{theorem}

\begin{proof}
This is very similar to, and at some stages slightly simpler than, the proof of Theorem
\ref{nownib1}. We make a brief remark on the necessary changes to part (4) to show
that $L/F$ is cyclic when $r$ and $\frakp$ are not coprime. A key point here is that
$F$ is a subfield of $K$. Observe that $\Gal(L/K) \cong (\Z/\ell\Z)$ and 
$\Gal(K/F) \cong (\Z/r\Z)$ where $r \neq \ell$, and $L/F$ is Galois. 
So $\Gal(L/F)$ is cyclic if and only if it is abelian, which is the case precisely when
the action of $\Gal(K/F)$ on $\Gal(L/K)$ is trivial. However, $\Gal(L/K)$ identifies
via class field theory with a quotient of the multiplicative group of $K$ at the prime
above $\frakp$. Due to total ramification in $K/F$, this residue field is the same 
as the residue field of $F$ at $\frakp$, so the action of $\Gal(K/F)$ on $\Gal(L/K)$ 
is indeed trivial.
\end{proof}

\begin{remark}
Both Theorem \ref{nownib2} and Corollary \ref{nownib1-cor} still hold when hypothesis
($\textrm{a}'$) is assumed instead of (a). Of course, Corollary \ref{nownib1-cor}
can itself be viewed as another weakening of hypothesis (a).
\end{remark}

\begin{definition}
Let $k \subset K \subset L$ be a tower of number fields. We say that
$L/K/k$ has \emph{disjoint ramification} if there is no finite prime $\mathfrak{p}$
that ramifies both in $K/k$ and $L/K$. (We already remarked on this abuse
of language in Definition \ref{abuse-ram-def}.)
\end{definition}

\begin{prop}\label{wnib-implies-disjoint-ram}
Let $L/k$ be an abelian extension of number fields with $[L:k]$ odd and $k$ totally real.
Let $K$ be an intermediate field such that $L/K$ is tame and $L/K$ has a WNIB. 
Suppose that for all prime divisors $\ell$ of $[L:K]$ we have $[K(\zeta_{\ell}):K]=\ell-1$,
and that at least one of the following conditions is satisfied:
\begin{enumerate}
\item $[L:K]$ is not divisible by $3$; or
\item for all primes $q$ dividing $[K:k]$, we have $\zeta_{q} \not \in L(\zeta_{3^{\infty}})$.
\end{enumerate}
Then $L/K/k$ has disjoint ramification.
\end{prop}

\begin{proof}
By assuming the contrary that there is a finite prime ramified in both $K/k$ and $L/K$,
it follows directly from Theorem \ref{nownib1} that conditions (a) and (b) each
give the desired conclusion.
\end{proof}

\begin{remark}
Proposition \ref{wnib-implies-disjoint-ram} can be modified in a number 
of ways by using the variants of Theorem \ref{nownib1} discussed above.
\end{remark}

\section{Splitting theorems for normal integral bases}\label{splitting-theorems}

Let $K$ be a number field and let $\Omega_{K}$ denote its absolute Galois
group. We fix a finite abelian group $G$. A $G$-extension $M/K$ is a commutative
$K$-algebra $M$ with a $G$-action, such that $M$ is a $G$-Galois extension
in the sense of Galois theory of commutative rings (see
\cite{greither} for an introduction), also known as a $G$-Galois algebra.
It is known that any such $M$ has the form $\ind_{G_{0}}^{G} M_{0}$,
where $M_{0}/K$ is a $G_{0}$-Galois extension in the usual sense
(i.e. $M_{0}$ is a field), $G_{0}$ is a subgroup of $G$, and as a $K$-algebra,
$\ind_{G_{0}}^{G} M_{0}$ is just a product of $[G:G_{0}]$ factors $M_{0}$.
(The ``$\ind$'' notation is useful for obtaining the $G$-action on the product.)
The field $M_{0}$ is called the \emph{core field} of the Galois algebra $M$.

The set H$(K,G)$ of all $G$-extensions $M/K$ modulo $G$-isomorphism 
carries the structure of an abelian group. The product of $M$ and $N$ is
given as follows:
$M \tensor_K N$ is a $G\times G$-extension of $K$ in the natural
way; let $D$ (the anti-diagonal) be the kernel of multiplication
$G \times G \to G$, so $(G\times G)/D$ is identified with $G$.
Then $M * N$ is (the class of) $(M \tensor_K N)^D$, with the
natural structure of $(G\times G)/D = G$-extension. (For this,
and more, see for example \cite{mcculloh-abelian}.)

There exists an isomorphism
\[
\Hone(\Omega_K, G) = \Hom(\Omega_K^{ab},G) \longrightarrow \hbox{\rm H}(K,G),
 \quad \phi \mapsto M_\phi
\]
with the following description: for surjective $\phi$, $M_\phi$ is
the fixed field of $K^{alg}$ under the kernel of $\phi$,
with the $G$-action resulting from $\Omega_K/\ker(\phi) \cong G$.
In general, let $G_0$ be the image of $\phi$; then $M_{0,\phi}$ is
defined as just explained, and $M_\phi$ is obtained by induction
from $G_0$ to $G$.

There are canonical subgroups 
$\Hone_{tame}(\Omega_K, G)$ and
 $\Hone_{unr}(\Omega_K, G)$ of  $\Hone(\Omega_K, G)$:
the subgroups afforded by tame (resp.\ unramified) extensions.
In terms of $G$-extensions, $M$ is tame (resp.\ unramified)
if and only if its core field is tame (resp.\ unramified). Using the alternative $\Hone$
description above, $\phi$ is tame (resp.\ unramified) if and only if it is trivial on all
higher ramification groups (resp.\ all inertia groups). 

The class invariant map
\[
\pic: \Hone_{tame}(\Omega_K, G) \longrightarrow \Pic(\O_K[G]), 
\]
sends $M$ to the class of the $\O_K[G]$-module $\O_M$.
(Note that again $\O_M$ can be described in terms of the core field:
it is $\ind_{G_0}^G \O_{M_0}$, or equivalently, the integral closure of 
$\O_K$ (or $\Z$) in $M$.) This is a homomorphism when restricted
to unramified extensions, but this is not the case in general.
However, we do have the following result (this is ascribed to McCulloh
by Brinkhuis; it also appears in \cite[p.225-226]{frohlich-book}). We say
that two $G$-extensions are arithmetically disjoint over $K$ if and only 
if their core fields are.

\begin{lemma}\label{weakmult}
If $M$ and $M'$ are arithmetically disjoint and
tame over $K$, then $$\pic(M * M') = \pic(M)\pic(M').$$
\end{lemma}

\begin{proof}
We have 
\[
\O_{M*M'} = \O_{(M\tensor M')^D} = (\O_{M\tensor M'})^D
   = (\O_M \tensor_{\O_K} \O_{M'})^D,
\]
where the third equality comes from the arithmetical disjointness.
Let $P = \O_M \tensor_{\O_K} \O_{M'}$. Then $P$ is a locally
free $\O_K[G\times G]$-module and is therefore cohomologically
trivial. Letting $I_D$ denotes the kernel of augmentation, we obtain
\[
P^D = \Norm_D\cdot P \cong P/\{x\in P: \Norm_D(x)=0\} 
= P/I_DP = P \tensor_{\O_K[G\times G]}  \O_K[G]
\]
as $\O_K[G\times G]$-modules. However, the last term is the
finest quotient module of $\O_M \tensor_{\O_K} \O_{M'}$
on which $D=\{(\sigma,\sigma^{-1}):\sigma\in G\}$ 
acts trivially, and this is simply the
tensor product $\O_M \tensor_{\O_K[G]} \O_{M'}$,
which has class $\pic(M)\pic(M')$ in $\Pic(\O_K[G])$.
\end{proof}

\begin{definition}
Let $k \subset K \subset L$ be a tower of number fields. We say that
$L/K/k$ is \emph{arithmetically split} if there exists an extension $L'/k$
such that $L=L'K$ and $L'$ is arithmetically disjoint from $K$ over $k$.
\end{definition}

\begin{theorem}\label{NIB-disjoint-split}
Let $L/K$ be a finite extension of number fields such that $L/\Q$ is abelian,
$L/K$ is tame of odd degree and $K$ is totally real. Then $L/K/\Q$ is arithmetically 
split if and only if $L/K$ has an NIB and $L/K/\Q$ has disjoint ramification.
\end{theorem}

\begin{remark}
If $n > 2$ is not a prime power, then $\Q(\zeta_{n})/\Q(\zeta_{n})^{+}$
is unramified at all finite primes (see \cite[Proposition 2.15]{wash})
and so $\Q(\zeta_{n})/\Q(\zeta_{n})^{+}/\Q$ has disjoint ramification.
Furthermore, $\Q(\zeta_{n})/\Q(\zeta_{n})^{+}$ has an NIB generated by $\zeta_{n}$
but $\Q(\zeta_{n})/\Q(\zeta_{n})^{+}/\Q$ is not arithmetically split. Therefore
the hypothesis that $[L:K]$ is odd cannot be completely removed from
Theorem \ref{NIB-disjoint-split}.
\end{remark}

\begin{proof}
(1) Suppose that $L/K/\Q$ is arithmetically split. Then $L'/\Q$ has an NIB by
the Hilbert-Speiser Theorem and so $L/K$ also has an NIB by Lemma \ref{arith-disjoint}.
Furthermore, it is clear that $L/K/\Q$ must have disjoint ramification.

(2) Suppose conversely that $L/K$ has an NIB and $L/K/\Q$ has disjoint ramification.
Let $\Omega = \Omega_{\Q}^{ab}$ and let $\Delta \subset \Omega$ be the group fixing $K$.
Let $G=\Gal(L/K)$ and $\phi \in \Hone(\Omega_K,G)$ be associated to the $G$-extension $L/K$. Then $\phi$ must factor through $\Delta$ because $L$ is abelian over $\Q$.
On the other hand, there is the following general fact: 
if $\psi \in \Hone(\Omega, G)$ belongs to the extension 
$M/\Q$ for some $M$ under the correspondence explained above, 
then $\psi|_{\Delta}$ belongs to the base-changed $G$-extension $K \tensor_{\Q} M/K$.
(Even if $M$ is a field, $K \tensor_{\Q} M$ need not be a field;
but it certainly is a $G$-Galois algebra. This is another advantage
of the formalism of Galois algebras.)

(3) Since the maximal abelian extension of $\Q$ is the
linearly disjoint compositum of all its inertia fields (each
being given in the form $\Q(\zeta_{p^\infty})$, $p$ prime),
the group $\Omega$ is the direct product of all the inertia
groups: $\Omega = \prod_p T_p$, with $p$ running over all primes.
For any set $\Sigma$ of rational primes, let $T_\Sigma = 
\prod_{p\in \Sigma} T_p \subset \Omega$. Now let $S$ be
the set of primes that ramify in $K/\Q$, and $S'$ its complement.
Then clearly $\Omega = T_S \times T_{S'}$. Furthermore, 
$T_{S'}$ is a subgroup of $\Delta$ (the group corresponding to $K$), 
because $L/K/\Q$ has disjoint ramification.
On the other hand, for every prime of $K$ over $p$, its inertia 
group in $\Q^{ab}/K$ is given by $T_p \cap \Delta$.

(4) Let $L/K$ be given by $\phi: \Delta \to G$. (Then $\phi$ is
onto since $L$ is a field.) We construct $\psi: \Omega \to G$
as follows:
\begin{eqnarray*}  \psi|_{T_{S'}} &=& \phi|_{ T_{S'}}; \\
 \psi|_{T_S} &=& 1_G. 
\end{eqnarray*}
Let $L'$ be the $G$-extension of $\Q$ attached to $\psi$
(so $L'=M_\psi$ in the notation used above). 
By construction, $L'$ is arithmetically disjoint from 
$K$ over $\Q$. In particular, the Galois algebra
$L_0 :=K \tensor_{\Q} L'$ over $K$ is a field, and
as said in (2), $L_0/K$ is attached to $\psi|_{\Delta}$. 

We now need an explicit description of the inverse of
a $G$-extension $N/K$ in H$(K,G)$: this is simply $N^{op}$,
which equals $N$ as a $K$-algebra, but $G$ acts through
the inverse map $\sigma \mapsto \sigma^{-1}$. We define
a new $G$-extension by setting
\[
M := L_0 * L^{op} 
\]
(product in H$(K,G)$).  Then $M$ is attached to the
difference $\alpha:=\psi|_\Delta - \phi$. Now $\alpha$ is trivial
on $T_{S'}$ by construction, and it is trivial on $T_p\cap \Delta$ 
for each $p\in S$, since
$\phi$ is trivial on $T_p\cap \Delta$ (assumption on ramification in $L/K$)
and $\psi$ is trivial on $T_p$ by definition. This means
precisely that $\psi$ is trivial on
{\it all\/} ramification groups in $\Delta$, that is,
$M/K$ is unramified.

(5) If $M/K$ is the trivial $G$-extension (equivalently: its
core field is just $K$), then $L_0$ and $L$ are the same as
$G$-extensions of $K$, in particular, they are the same as
$K$-algebras. Hence
$L_0=L$ considered as subfields of $\Q^{ab}$, and we recall
that $L'$ is arithmetically disjoint from $K$ over $\Q$. 
Thus it now suffices to show that the other case, i.e., $M/K$ nontrivial, 
is impossible.

(6) The class invariant map is compatible with induction,
so if $M_0$ is the core field of $M$, then $\pic(M)= \ind_{G_0}^G \pic(M_0)$.
Let $j: G \to G$ be the inversion map on $G$;
by functoriality $j$ induces an involution $j_*$ on $\Pic(\O_K[G])$.
In the following, we let $X^{-}$ denote the subgroup of all
$x \in X$ having odd order satisfying $j_*x = -x$.
We make two claims:

\begin{enumerate}
\renewcommand{\labelenumi}{(\Alph{enumi})}
\item $\pic(M_0) \in \Pic(\O_K[G_0])^{-}$; and
\item induction induces an injection
$\Pic(\O_K[G_0])^- \to \Pic(\O_K[G])^-$.
\end{enumerate}

We assume the validity of these claims and return to their proofs later. 
Since $M_{0}/K$ is unramified, $[M_{0}:K]$ is odd
and $K$ is totally real, Theorem \ref{br-unram} 
(\cite[Theorem 1]{brink}) due to Brinkhuis shows that $M_{0}/K$ has no NIB, i.e.,
$\pic(M_{0})$ is nontrivial. Hence, by the two claims, $\pic(M)$ is also nontrivial. 
Now we have 
\[ 
M = L_0 * L^{op},
\] 
which is equivalent to
\[
L = L_0 * M^{op}.
\]
By the Hilbert-Speiser Theorem, $L'/\Q$ has an NIB since it is tame abelian
(this follows from the tameness of $L/K$ and the construction of $L'$).
By Lemma \ref{arith-disjoint}, it follows that $L_0/K$ also has an NIB 
since $L'$ is arithmetically disjoint from $K$ over $\Q$. Furthermore,
we started from the assumption that $L/K$ has an NIB.
Therefore Lemma \ref{weakmult} applied to $L = L_0 * M^{op}$
leads to an immediate contradiction.

(7) It remains to establish claims (A) and (B).

\smallskip

Proof of (A):
It follows from the fact that $\pic$ is a homomorphism on unramified extensions
that  $|G_0|\pic(M_0)$ is trivial. By functoriality, $\pic(j_*(M_0)) = j_*(\pic(M_0))$.
(Here $j_*(M_0)$ is the same algebra as $M_0$, with inverted
action of $G$.) But $j_*(M_0)$ happens to also be the inverse
of $M_0$ in H$(K,G_0)$, so again because $\pic$ is a homomorphism
on unramified extensions, $\pic(j_*(M_0)) = -(\pic(M_0))$.

\smallskip

Proof of (B): This is considerably harder.
We write $U$ for $G_0$.
The main obstacle is that $S:=\O_K[G]$ is not a Galois extension
of the ring $R:=\O_K[U]$, so Galois cohomology cannot be
used to calculate $\Ker(\Pic(R) \to \Pic(S))$. Instead,
we use faithfully flat descent. Of course, $S$ is faithfully
flat (even free) over $R$. The first Amitsur cohomology of
the multiplicative group $\Hone_{A}(S, \mathbb{G}_m)$ is canonically
isomorphic to $\Ker(\Pic(R) \to \Pic(S))$. We recall the definition:
there is a complex
\[
S^{\times} \stackrel{\partial_{1}} \longrightarrow (S\tensor_R S)^{\times} 
\stackrel{\partial_{2}} \longrightarrow (S\tensor_R S\tensor_R S)^{\times},
\]
where $\partial_1$ sends $s$ to $s\tensor s^{-1}$,
and $\partial_2$ sends $u$ to 
$u_1 \cdot u_2^{-1} \cdot u_3$. Here $u_1, u_2, u_3 \in 
(S\tensor_R S\tensor_R S)^{\times}$ denote the respective images of
$u$ under the maps defined on $(S\tensor_R S)^{\times}$, 
putting in a 1 on the left, in the middle and on the right,
so, for example, $u_2(s \tensor t) = s \tensor 1 \tensor t$.

The Amitsur cohomology group is now the cohomology of this
complex at the middle. We will show that the odd minus part
of this is trivial. This heavily relies on an important
result of Lenstra (see \cite[p.159]{brink}): 
If $K$ is totally real (and this is the
case in our situation), then the ``minus part'' 
$(\O_K[\Gamma]^{\times})^{1-j}$ of the
unit group of the group ring  of
any abelian odd order group consists only of $\pm \Gamma$ itself.
It is obvious that $S\tensor_R S$ can be identified with
the group ring $\O_K[G^{(2)}]$, where $G^{(2)}$ is the
pushout of $G$ with itself over $G_0$ (more explicitly:
$G \times G$ factored out by all $(z,z^{-1})$ with $z\in G_0$),
and a similar statement holds for the triple tensor product.
We exponentiate all terms in the last complex with $1-j$,
and obtain (we neglect $\pm1$):
\[
G   \longrightarrow   G^{(2)}  \longrightarrow  G^{(3)},
\]
and the maps are in close analogy to the previous maps:
$x \in G$ goes to $(x,x^{-1}) \in G^{(2)}$, and $(x,y) \in G^{(2)}$
goes to $(x,y,1) (x^{-1}, 1, y^{-1}) (1,x,y) \in G^{(3)}$. 
The cohomology of this new complex then is just the
minus part of the cohomology of the old one, at least
in the odd part.
It is now just an exercise to show that this new complex is exact,
so its middle cohomology is trivial, and this means that
the odd minus part of the Amitsur cohomology is trivial, as
required.
\end{proof}

\begin{theorem}\label{NIB-split}
Let $L/K$ be a finite extension of number fields such that $L/\Q$
is abelian, $L/K$ is tame and $[L:\Q]$ is odd. Suppose that either
\begin{enumerate}
\item $[L:K]$ is not divisible by $3$; or
\item for all primes $q$ dividing $[K:\Q]$, we have $\zeta_{q} \not \in L(\zeta_{3^{\infty}})$.
\end{enumerate}
Then $L/K$ has an NIB if and only if $L/K/\Q$ is arithmetically split.
\end{theorem}

\begin{proof}
Suppose that $L/K/\Q$ is arithmetically split. Then $L'/\Q$ has an NIB by
the Hilbert-Speiser Theorem and so $L/K$ also has an NIB by Lemma \ref{arith-disjoint}.

Suppose conversely that $L/K$ has an NIB. There exist intermediate fields 
$L_{1}, \ldots, L_{r}$ such that $L$ is equal to the compositum $L_{1} \cdots L_{r}$
and for (not necessarily distinct) odd primes $\ell_{1}, \ldots, \ell_{r}$ we have 
$\Gal(L_{i}/K) \cong (\Z/\ell_{i}^{s_{i}}\Z)$ for some $s_{i} \geq 1$. 
By Lemma \ref{trace-down}, each extension $L_{i}/K$ has an NIB. 
Suppose that each $L_{i}/K/\Q$ is arithmetically split, i.e., there exist fields $L_{i}'$
each arithmetically disjoint from $K$ over $\Q$ such that $L_{i} = L_{i}' \cdot K$.
Let $L' = L_{1}' \cdots L_{r}'$. It is straightforward to check that $L = L' \cdot K$
and that $L'$ is arithmetically disjoint from $K$ over $\Q$. Hence $L/K/\Q$
is arithmetically split, as desired. Thus we are reduced to the case where 
$L/K$ is cyclic and $[L:K]=\ell^{s}$ for some odd prime $\ell$ and some $s \geq 1$.

Observe that $L$ is linearly disjoint from $K(\zeta_{\ell})$ over $K$ since $[L:K]=\ell^{s}$ 
and $[K(\zeta_{\ell}):K]$ divides $\ell - 1$. Furthermore, $L/K$ is tamely ramified and 
$K(\zeta_{\ell})/K$ is only ramified at primes above $\ell$, if at all. 
Therefore $L$ is in fact arithmetically disjoint from $K(\zeta_{\ell})$ over $K$, and so 
$L(\zeta_{\ell})/K(\zeta_{\ell})$ also has an NIB by Lemma \ref{arith-disjoint}.

Suppose for a contradiction that $L/K/\Q$ does not have disjoint ramification, i.e., 
there exists a prime $p$ that ramifies in both $K/\Q$ and $L/K$. 
(Note that $p \neq \ell$ because $L/K$ is tamely ramified.) 
It is straightforward to see that $p$ ramifies in both $K(\zeta_{\ell})/\Q$ and
$L(\zeta_{\ell})/K(\zeta_{\ell})$.

We now use the theory of Dirichlet characters as described in \cite[Chapter 3]{wash}.
For $n \in \N$, let $X^{(n)}$ denote the group of Dirichlet characters corresponding
to $\Q(\zeta_{n})$. Let $m\ell^{t}$ be the conductor of $L(\zeta_{\ell})$ over $\Q$ where 
$\ell \nmid m$. Let $X$ be the group of Dirichlet characters corresponding to $L(\zeta_{\ell})$ 
and let $Y$ be the Sylow-$\ell$ subgroup of $X^{(\ell^{t})}$. Then we have
$X^{(\ell)} \subseteq X \subseteq X^{(m)} \times Y \times X^{(\ell)}$
and so $X$ is of the form $Z \times X^{(\ell)}$. Let $L'$ be the field corresponding
to $Z$ and construct $K'$ analogously, i.e., ``remove $\Q(\zeta_{\ell})$''.
By \cite[Theorem 3.5]{wash}, $p$ is ramified in both $L'/K'$ and $K'/\Q$.
Moreover, $[K':\Q]$ is odd since $[K':\Q]$ divides $[K:\Q]$ divides $[L:\Q]$,
and so $K'$ is totally real.
Hence $L'/K'$ satisfies the hypotheses of Theorem \ref{nownib1}, and so
applying Corollary \ref{nownib1-cor} shows that $L(\zeta_{\ell})/K(\zeta_{\ell})$
does not have a WNIB. However, this is a contradiction because
$L(\zeta_{\ell})/K(\zeta_{\ell})$ has an NIB.  We therefore conclude that $L/K/\Q$ 
must in fact have disjoint ramification.
Since $[K:\Q]$ is odd and $K/\Q$ is Galois, $K$ is totally real and so 
Theorem \ref{NIB-disjoint-split} now gives the desired result.
\end{proof}

\begin{remark}
Condition (b) of Theorem \ref{NIB-split} can be weakened as follows
(we use notation from the proof): for each $i$ with $\ell_{i}=3$ and each
prime $q$ dividing $[K:\Q]$, we have $\zeta_{q} \not \in L_{i}(\zeta_{3^{\infty}})$.
\end{remark}

\begin{remark}
The results of Brinkhuis stated in the introduction can be easily recovered 
in the special setting of Theorem \ref{NIB-split}. It is straightforward to see that extensions satisfying 
the hypotheses of Theorem \ref{br-non-split} cannot be arithmetically split, and
Theorem \ref{br-unram} for $L$ absolutely abelian becomes a consequence of the fact that 
there are no nontrivial unramified extensions of $\Q$.
Of course, Brinkhuis's results also hold in a much more general setting 
and it should be noted that the proof of Theorem \ref{NIB-split} relies
on Theorem \ref{br-unram}.
\end{remark}

\section{Acknowledgments}\label{acknowledgments}

The authors are grateful to the Deutscher Akademischer Austausch Dienst (German Academic Exchange Service) for a grant allowing the second named author to visit the first for the 2006-07 academic year, thus making this collaboration possible. Furthermore, the authors are indebted to the referee for several corrections and helpful comments.

\bibliography{WNIB-Bib}{}
\bibliographystyle{amsalpha}

\end{document}